\documentclass{article}
\usepackage[utf8]{inputenc}
\usepackage{enumerate}
\usepackage{amsmath,amssymb, bm}
\usepackage{amsthm}
\title{On the hard Lefschetz theorem\\ for pseudoeffective line bundles}
\author{Xiaojun WU}
\date{\today}
\newtheorem{mythm}{Theorem}
\newtheorem{mylem}{Lemma}
\newtheorem{myprop}{Proposition}
\newtheorem{mycor}{Corollary}
\newtheorem{mydef}{Definition}
\newtheorem{myrem}{Remark}
\begin{document}
\def\cP{\mathcal{P}}
\def\Z{\mathbb{Z}}
\def\Q{\mathbb{Q}}  \def\C{\mathbb{C}}
 \def\R{\mathbb{R}}
 \def\N{\mathbb{N}}
 \def\ii{\mathrm{i}}
  \def\d{\partial}
 \def\dbar{{\overline{\partial}}}
\def\dzbar{{\overline{dz}}}
\def \ddbar {\partial \overline{\partial}}
\def\cK{\mathcal{K}}
\def\cE{\mathcal{E}}  \def\cO{\mathcal{O}}
\def\P{\mathbb{P}}
\def\cI{\mathcal{I}}
\def \loc{\mathrm{loc}}
\def \log{\mathrm{log}}
\def \cC{\mathcal{C}}
\bibliographystyle{plain}
\def \deg{\mathrm{deg}}
\def \RHS{\mathrm{RHS}}
\def \Ker{\mathrm{Ker}}
\def \Gr{\mathrm{Gr}}
\def \liminf{\mathrm{liminf}}
\def \ker{\mathrm{Ker}}
\bibliographystyle{plain}
\long\def\forget#1{}
\def\myparagraph#1{\vskip 6pt plus 2pt minus 2pt}

\maketitle
\begin{abstract}
In this note, we obtain a number of results related to the hard Lefschetz theorem for pseudoeffective line bundles, due to Demailly, Peternell and Schneider. Our first result states that the holomorphic sections produced by the theorem are in fact parallel, when viewed as currents with respect to the singular Chern connection associated with the metric. Our proof is based on a control of the covariant derivative in the approximation process used in the construction of the section. 
Then we show that we have an isomorphsim between such parallel sections and higher degree cohomology. As an application, we show that the closedness of such sections induces a linear subspace structure on the tangent bundle. Finally, we discuss some questions related to the optimality of the hard Lefschetz theorem.
\end{abstract}

\section{Introduction}

In this note, we establish a closedness and harmonicity result that complements the hard Lefschetz theorem for pseudoeffective line bundles proved in \cite{DPS01}. By following the arguments of the above paper, we show that the sections provided by the proof are in fact parallel, when viewed as currents with respect to the singular Chern connection of the metric. The first difficulty is to define the covariant derivative for such singular metrics, since in general the wedge product of two currents is not always well-defined. Another difficulty is to control the covariant derivative in the approximation process employed in the original proof.
\myparagraph{}
Let $X$ be a compact K\"ahler $n$-dimensional manifold, equipped with a K\"ahler metric, i.e.\ a positive definite Hermitian $(1,1)$-form $\omega=\ii\sum_{1\le j,k\le n}
\omega_{jk}(z)\,dz_j\wedge d\overline z_k$ such that $d\omega=0$. By definition 
a holomorphic line bundle $L$ on $X$ is said to be pseudoeffective if there exists a singular hermitian metric $h$ on $L$, given by $h(z)=e^{-\varphi(z)}$ with respect to a local trivialization $L_{|U}\simeq U\times\C$,
such that the curvature form
$$
\ii\Theta_{L,h}:=\ii\ddbar\varphi
$$
is (semi)positive in the sense of currents, i.e.\ $\varphi$ is locally
integrable and \hbox{$i\Theta_{L,h}\ge 0\,$}: in other words, the weight
function $\varphi$ is plurisubharmonic (psh) on the corresponding trivializing
open set~$U$. In this trivialization, if the metric is in fact smooth, the (1,0) part of the covariant derivative with respect to the associated Chern connection is given in the form:
$$\d_h=\d + \d \varphi \wedge \bullet,$$
and the total connection is $d_h=\d_h+\dbar$. An important fact is that $\d_h$ and $d_h$ still make sense for an arbitrary singular metric $h$ as above. Another basic concept relative to a singular metric is the notion of {\it multiplier ideal
sheaf}, introduced in \cite{Nad90}.

\begin{mydef} To any psh function $\varphi$ on an open subset $U$ of
a complex manifold~$X$, one associates the ``multiplier ideal sheaf'' $\cI(\varphi)\subset\cO_{X|U}$ of germs of
holomorphic functions $f\in\cO_{X,x}$, $x\in U$, such that $|f|^2e^{-\varphi}$
is integrable with respect to the Lebesgue measure in some local
coordinates near~$x$. We also define the global multiplier ideal sheaf $\cI(h)\subset\cO_X$ of a hermitian metric $h$ on $L\in\mathrm{Pic}(X)$ to be equal to $\cI(\varphi)$ on any open subset $U$ where $L_{|U}$ is trivial and $h=e^{-\varphi}$. In such a definition, we may in fact assume $\ii\Theta_{L,h}\geq -C\omega$,
i.e.\ locally $\varphi=\hbox{psh}+C^\infty$, we say in that case that
$\varphi$ is {\rm quasi-psh}.
\end{mydef}
The interest of considering quasi-psh functions is that on a compact manifold global psh functions are constant, while the space of quasi-psh functions is infinite dimensional. Among them, functions with analytic singularity will be of special concern for us.
With this notation, the following bundle valued generalization of the hard Lefschetz theorem has been established in \cite{DPS01}. The proof uses the natural $L^2$-resolution of the sheaf $\Omega_X^n \otimes L \otimes \cI (h)$.
\begin{mythm}{\rm(\cite{DPS01})}\label{dps-th}
Let $(L,h)$ be a pseudo-effective line bundle on a
compact K\"ahler manifold $(X,\omega)$ of dimension $n$, let
$\Theta_{L,h}\ge 0$ be its curvature current and $\cI(h)$ the
associated multiplier ideal sheaf. Then, the wedge multiplication
operator $\omega^q\wedge\bullet$ induces a surjective morphism
$$
\Phi^q_{\omega,h}:
H^0(X,\Omega_X^{n-q}\otimes L\otimes\cI(h))\longrightarrow
H^q(X,\Omega_X^n\otimes L\otimes\cI(h)).
$$
\end{mythm}
The special case when $L$ is nef is due to Takegoshi \cite{Tak97} (for the definition of nef in analytic setting, cf. [DPS94]). An even
more special case is when $L$ is semipositive, i.e.\ $L$ possesses a
smooth metric with semipositive curvature.
In that case, the multiple ideal sheaf $\cI(h)$ coincides with $\cO_X$
and we get the following consequence already observed by Enoki \cite{Eno93}
and Mourougane \cite{Mou95}.

\begin{mycor}
Let $(L,h)$ be a semipositive line bundle
on a compact K\"ahler mani\-fold $(X,\omega)$ of dimension~$n$. Then, the wedge
multiplication operator $\omega^q\wedge\bullet$ induces a surjective morphism
$$
\Phi^q_\omega:H^0(X,\Omega_X^{n-q}\otimes L)\longrightarrow
H^q(X,\Omega_X^n\otimes L).
$$
\end{mycor}
It should be observed that although all objects involved in Theorem 1
are algebraic when $X$ is a projective manifold, there is no known
algebraic proof of the statement; it is not even clear how to define
algebraically $\cI(h)$ in the case when $h=h_{\min}$ is a metric
with minimal singularity. The classical hard Lefschetz theorem is the case when $L$ is trivial or unitary flat; then $L$ has a (real analytic) metric $h$ of curvature equal to $0$, whence
$\cI(h) =\cO_X$.

In the pseudoeffective case, the Lefschetz morphism is in general no longer injective as in the classical hard Lefschetz theorem. An obvious counterexample can be obtained by taking $L=mA$ where $A$ is an ample divisor, so that $h^0(X,\Omega_X^{n-q}\otimes L) \sim Cm^n$ for $m$ big enough, but $h^q(X,\Omega_X^n\otimes L)=0$ if $q>0$.
We will show that one can again recover an isomorphism by replacing the left hand side with the space of parallel sections with respect to the singular metric.

The proof of the hard Lefschetz is obtained by constructing directly a preimage for any given element in $H^q(X,\Omega_X^n\otimes L\otimes\cI(h))$.
This is done by taking a weak limit of some subsequence of a bounded sequence in a suitable Hilbert space, using the fact that every bounded sequence of a Hilbert space admits a weakly convergent subsequence; notice however that there are no reason for these weak limits to be unique.
One can also view the proof of the hard Lefschetz theorem as the construction of an inverse operator  
$$H^q(X,\Omega_X^n\otimes L\otimes\cI(h))
\longrightarrow H^0(X,\Omega_X^{n-q}\otimes L\otimes\cI(h)),
$$
although it is not a priori obvious that the preimages can be chosen to depend linearly on the given classes in $H^q$. This question is more or less equivalent to asking whether there exists a natural subspace of $H^0(X,\Omega_X^{n-q}\otimes L\otimes\cI(h))$ such that the restriction of the Lefschetz morphism is an isomorphism.
\myparagraph{}
In the classical case $L=\cO_X$, one can observe that any section $u \in H^0(X, \Omega_X^{n-q})$ satisfies the additional condition $du=d_{h_0}u=0$. This is easily seen by Stokes formula, which implies
$$\int_X i du \wedge \overline{du} \wedge \omega^{q-1}=\int_X \{ du, du\}_{h_0} \wedge \omega^{q-1}=0,$$
where $h_0$ is the trivial smooth metric on $\cO_X\,$; in that formula (as well as in the rest of this paper), given a hermitian metric $h$, we denote by $\{u,v\}_h$ the natural sesquilinear pairing
$$\cC^{\infty}(M, \wedge^pT^*_X \otimes L) \times \cC^{\infty}(M, \wedge^qT^*_X \otimes L) \rightarrow \cC^{\infty}(M, \wedge^{p+q}T^*_X )$$
$$(u,v) \mapsto \{u,v\}_h$$
given by
$$\{u,v\}_h=\sum_{\lambda,\mu}i u_{\lambda} \wedge \bar{v}_{\mu} \langle e_{\lambda}, e_{\mu} \rangle_h$$
where $u=\sum u_{\lambda} \otimes e_{\lambda}$, $v=\sum v_{\mu} \otimes e_{\mu}$.
Another proof relies on the observation that $\dbar u=\dbar^{*} u=0$ (the second equality holds since $u$ is of bidegree $(n-q,0)$), whence $\Delta_{\dbar}u=0=\Delta_{\d} u$ by the K\"ahler identities. As~a consequence, we have $\d u=\d^* u=0$, and so $du=0$.

More generally, the proof of the hard Lefschetz theorem in \cite{DPS01} is obtained by constructing preimages as limits of forms given by the pointwise Lefschetz isomorphism. One then deals with a sequence of harmonic representatives of a given class in $H^q(X, K_X \otimes L \otimes \mathcal{I}(h))$, with respect to approximated, less singular, hermitian metrics $h_{\varepsilon}$.
It is thus natural to wonder whether the holomorphic sections provided by Theorem \ref{dps-th} also satisfy some sort of closedness property in the case of arbitrary pseudoeffective line bundles. In fact, we are going to prove that these sections are parallel with respect to the (possibly singular) Chern connection associated with the metric $h$; the proof employs similar arguments, but with the additional difficulty that one has to deal with non smooth metrics.
\begin{mythm}All holomorphic sections produced by Theorem \ref{dps-th} are parallel with respect to the Chern connection associated with the singular hermitian metric $h$ on~$L$.
\end{mythm}
More precisely, as $h$ can be singular, this means that in local coordinates, any such holomorphic section $s \in H^0(X,\Omega_X^{n-q}\otimes L\otimes\cI(h))$ satisfies
$$\d_h s=\d s +\d \varphi \wedge s =0$$
in the sense of currents. Since $\dbar s=0$, we conclude that $d_hs=\d_hs+\dbar s=0$. This property can be expressed by saying that the section $s$ is parallel with respect to~$d_h$. 

Now, let us consider the harmonicity. 
Assume first that the metric is semipositive (i.e. a smooth metric with positive Chern curvature).
By computing $\dbar(\d_hs)=0$, we get $\dbar\d \varphi\wedge s=0$, hence
$$i\Theta_{L,h} \wedge s=0.$$
As $\Delta_{\dbar}s=0$ ($s$ is a holomorphic section and $\dbar^*s=0$ by a bidegree consideration), the Kodaira-Nakano identity implies
$$\Delta_{\dbar}s-\Delta_{\d_h}s=[i\Theta_{L,h},\Lambda]s=i\Theta_{L,h} \Lambda s-\Lambda i\Theta_{L,h}s=-\Lambda i\Theta_{L,h}s=0,$$
by the fact that $\Lambda s=0$. Therefore $\Delta_{\d_h}s=0$. 
Since the metric is smooth, this is equivalent to the fact that
$\d_h s=0$ and $\d_h^* s=0$. If the metric is singular, we still have
$$i\Theta_{L,h} \wedge s=0$$
by the same arguments.
However, in the latter case, although the operator $\d_h$ is still a densely defined operator on $L^2(X,\Omega_X^{n-q}\otimes L,h)$ (cf. Remark 1), it is difficult to give an explicit expression of his Hilbert adjoint $\d_h^*$.
There may exist the boundary condition on the domain of $\d_h^*$ caused by integration by parts, while the singular part of a general positive singular metric could have very difficult topology.
Thus it is difficult to discuss the Hilbert adjoint $\d_h^*$ in general.
Nevertheless, the fact that the section is parallel with respect to the singular metric is sufficient to characterize the preimage of the wedge multiplication operator in the hard Lefschetz theorem.
\begin{mythm}
Let $(L,h)$ be a pseudo-effective line bundle on a
compact K\"ahler manifold $(X,\omega)$ of dimension $n$, let
$\Theta_{L,h}\ge 0$ be its curvature current and $\cI(h)$ the
associated multiplier ideal sheaf. Then,
the wedge multiplication operator $\omega^q \cdot \bullet$ induces a
linear isomorphism
$$
\Phi^q_\omega:H^0(X,\Omega_X^{n-q}\otimes L) \cap \Ker(\d_h) \longrightarrow
H^q(X,\Omega_X^n\otimes L).
$$
\end{mythm}
\myparagraph{}
In section~4, as a geometric application, we use the closedness property of the holomorphic sections produced by the hard Lefschetz theorem to derive the existence of a ``singular foliation'' of $X$ (in fact a linear subspace structure of $T_X$).

\begin{mythm}
Assume that $v \in H^0(X, \Omega_X^{n-q} \otimes L \otimes \cI(h)), q \geq 1$ is a parallel section with respect to the singular metric $h$.
In particular a section constructed by the hard Lefschetz theorem is such a section.  The interior product with $v$ gives an $\cO_X$- morphism (which is well defined throughout $X$)
$$F_v:T_X \to \Omega_X^{n-q-1} \otimes L$$
$$X \mapsto \iota_X v.$$
The kernel of $F_v$ defines an integrable coherent subsheaf of $\cO(T_X)$,
i.e.\ a holomorphic foliation.
\end{mythm}
At the end of section 4, we show by a concrete example indicated to the author by Professor Andreas H\"oring that for a general preimage, instead of the one constructed by the hard Lefschetz theorem, the above process does not necessarily induce a foliation.
In fact, the kernel of $F_v$ defined in the theorem defines a foliation if and only if $v$ is a parallel section.
\paragraph{}
Finally, in the last sections of this work, we discuss the optimality
of the multiplier ideal sheaf $\cI(h)=\cI(\varphi)$ involved in the
hard Lefschetz theorem.  Demailly, Peternell and Schneider already
showed in \cite{DPS01} that one cannot omit the ideal sheaf even
when $L$ is taken to be nef, and gave a couterexample when $L=-K_X$ is
the anticanonical bundle. However, it might still be possible in some
cases to ``improve'' the ideal sheaf, for instance to replace it with
$\lim_{\delta\to 0_+}\cI((1-\delta)\varphi)\supset \cI(\varphi)$. When
$\varphi$ has analytic singularities, it may happen that the
inclusion be strict, but in general the limit need not even be a coherent
sheaf (see section 5). 
The abundance conjecture and the nefness of $L=K_X$ would imply the
semiampleness of $L$, so in that case, the ideal sheaf is definitely
not needed. 
For the general case, this seems to be a difficult problem.
Some discussions of these issues are conducted in section 6.\vskip4mm

\textbf{Acknowledgement.} I thank Jean-Pierre Demailly, my PhD supervisor, for his guidance, patience and generosity. 
I am indebted to Chen-Yu Chi, Andreas H\"oring and Dano Kim for very helpful suggestions and comments on earlier drafts of this paper.
I would also like to express my gratitude to colleagues of Institut Fourier for all the interesting discussions we had. During the course of this research, my work has been supported by a Doctoral Fellowship AMX attributed by \'Ecole Polytechnique and Ministère de l'Enseignement Supérieur et de la Recherche et de l’Innovation, and I have also benefited from the support of the European Research Consortium grant ALKAGE Nr.~670846 managed by J.-P.~Demailly.

\section{Definition of the covariant derivative}
In this section, we consider a pseudoeffective line bundle $(L,h)$ on a K\"ahler (non necessarily compact) manifold $(Y,\omega)$ where $h(z)=e^{-\varphi(z)}$ with respect to a local trivialization $L_{|U}\simeq U\times\C$ and $\omega$ is smooth. We denote by
$|~~|=|~~|_{\omega,h}$ the pointwise hermitian norm on
$\Lambda^{p,q}T^\star_Y\otimes L$ associated with $\omega$ and $h$,
and by $\Vert~~\Vert=\Vert~~\Vert_{\omega,h}$ the global $L^2$ norm
$$
\Vert u\Vert^2 = \int_Y |u|^2 dV_\omega \qquad
\hbox{where}\quad  dV_\omega={\omega^n\over n!}.
$$
Recall that since $\varphi$ is a quasi-psh function on $U$, its derivative $d\varphi$ belongs to $L^p_{\loc}(U)$ with respect to Lebesgue measure for every $p < 2$ (cf. e.g. Theorem 1.48 in \cite{GZ17}). This regularity is optimal since on $\C$, the psh function $\log|z|$ has a derivative not in $L^2_{\loc}(\C)$. We fix a smooth reference metric $h_0$ on $L$ (not necessarily semipositive) from which we can view any other singular metric as given by $h=h_0e^{-\varphi}$ where $\varphi$ is a quasi-psh function defined on $Y$. In general, for $u \in L^2_{\loc}(U,\Lambda^{p,q}T^\star_Y\otimes L,\omega,h_0)$, $\d \varphi \wedge u$ is not a priori well defined as a form with coefficients in $L^1_{\loc}(U,\Lambda^{p+1,q}T^\star_Y\otimes L, \omega,h_0)$ (with respect to the Lebesgue measure), at least if we make a naive use of the Cauchy-Schwarz inequality to get a current on $U$. 
(Note that in this case, $\d \varphi \in L^1_{\loc}(U,\Lambda^{p+1,q}T^\star_Y\otimes L, \omega,h_0)$ \emph{is however} a current on $U$.)
\myparagraph{}
We can overcome this problem in our proof, because in the construction of sections in the proof of the bundle valued hard Lefschetz theorem, this type of product can always be defined. In fact we always have additional assumptions on either $u$ or $\varphi$, as we will see next, and this will be enough to prove our main theorem. At the end of this section, we prove that the wedge product $\partial\varphi\wedge u$ is closed with respect to the $L^2$ topology when $\varphi$ is any psh function and $u$ is in $L^2_{\loc}(e^{-\varphi})\,$; this will be used in the following section.
\myparagraph{}
In the sequel, we will make use two types of such wedge products. The first type is when $u$ is holomorphic, so that the coefficients of $u$ are in fact bounded on any compact set, hence in $L^{\infty}_{\loc}$, thus $\d \varphi \wedge u$ has coefficients in 
$$
L^1_{\loc}(U,\Lambda^{p,q}T^\star_Y\otimes L, \omega,h_0) \times L^{\infty}_{\loc}(U,\Lambda^{1,0}T^\star_Y\otimes L, \omega,h_0) \subset L^1_{\loc}(U, \Lambda^{p+1,q}T^\star_Y\otimes L, \omega,h_0).
$$ 
Moreover, if $\varphi_i$ a sequence of quasi-psh functions such that $\varphi_i \to \varphi$ in $L^1_{\loc}(U,\omega, h_0)$, we have $\d \varphi_i \to \d \varphi$ in $L^1_{\loc}(U,\Lambda^{1,0}T^\star_Y\otimes L, \omega,h_0)$ hence $\d \varphi_i \wedge u\to \d \varphi \wedge u $ in $L^1_{\loc}(U,\Lambda^{p+1,q}T^\star_Y\otimes L, \omega,h_0)$, which implies in particular the weak convergence as currents (cf. e.g. theorem 1.48 in \cite{GZ17}).
\myparagraph{}
The second type is when $\varphi$ is an arbitrary psh function, taken as a local weight function of~$h$, and $u \in L^2_{\loc}(U,\Lambda^{p,q}T^\star_Y\otimes L, \omega,h)$. 
\forget{
To understand what happens, we start by the case when $\varphi$ has analytic singularities, although thi consideration is not necessary for the proof of general case.
Suppose that $\varphi$ has analytic singularities along a simple normal crossing divisor, i.e.\ in some coordinates,
 $$\varphi=c\,\log|z_1^{a_1} ...z_n^{a_n}|+\mathrm{C} ^{\infty}.$$
We only need to check the current is well defined near a point in Sing($h$), a situation which happens only in case $c > 0$.
When $u \in L^2_{\loc}(U,\Lambda^{p,q}T^\star_Y\otimes L, \omega,h)$, we have to show that $\d \varphi \wedge u$ is locally integrable with respect to the Lebesgue measure, and without loss of generality, we can suppose that the section is integrable on $U$, and not only on every compact in $U$, i.e.\
$$\int_U |\d \varphi \wedge u|_{\omega,h_0} dV_{\omega} < \infty.$$
It is true since

$$\leq C (\int_U |z_1^{a_1} ...z_n^{a_n}|^{c}|\sum \frac{a_i}{2} \frac{dz_i}{z_i}|^2_{\omega,h_0}dV_{\omega})^{\frac{1}{2}} (\int_U |u|^2_{\omega,h}dV_{\omega})^{\frac{1}{2}}$$
$$\leq C (\int_U |z_1^{a_1} ...z_n^{a_n}|^{c}\sum \frac{a_i^2}{4|z_i|^2}\prod idz_i \wedge \dzbar_i)^{\frac{1}{2}} (\int_U |u|^2_{\omega,h}dV_{\omega})^{\frac{1}{2}}.$$
Since $U$ is a local coordinate chart, we can suppose $U$ to be a polydisc $\prod D(0,r_j)$. The integrability of the first term in the integral is given by for any $j$ such that $a_j >0$,
$$\int_U |z_1^{a_1} ...z_n^{a_n}|^{c} \frac{a_j^2}{4|z_j|^2}\prod idz_i \wedge \dzbar_i \leq C_j \int_{0}^{r_j} r^{2a_jc-1} < \infty$$
since $c a_j-1 >-1$.
By assumption the second term in the integral is finite, so the product is finite. 
\myparagraph{}
If $\varphi$ has analytic singularities, there exists a modification of $\mu: \tilde{Y} \rightarrow Y$ such that $\mu^*(\cI(h))$ is an invertible sheaf associated to a simple normal crossing divisor, thanks to Hironaka's desingularization theorem \cite{Hir64}. Since we consider only local integrability of functions up to modification (by definition, a modification is a biholomorphism outside of proper analytic sets) and since analytic subsets are of Lebesgue measure zero, singularities are irrelevant with respect to integration. Therefore, we can reduce the general case by using a modification that converts the singular sets involved into simple normal crossing divisors.

For the general case where $\varphi$ is arbitrary psh function.} It is enough to prove that $\int_K |e^{\frac{\varphi}{2}} \d \varphi|^2_{\omega,h_0}dV_{\omega}$ is finite for any compact set $K \Subset U$. After shrinking $U$ into a smaller relatively compact open subset, we can suppose that $\varphi \leq C$ for some $C >0$, and also that there exists a non increasing sequence of smooth psh functions $\varphi_{\varepsilon_\nu}$ converging to $\varphi$ in $L^1(U)$ as $\varepsilon_\nu \to 0$.
The smooth psh function sequence can be obtained by taking a convolution with radially symmetric approximations of the Dirac measure. The upper bound is obtained by the maximum principle. The same is true for $\varphi_{\varepsilon_1}$.
In particular, $e^{\varphi} \in L^1(U)$. We prove that $e^{\varphi} \in PSH(U)$. Up to a subsequence, $e^{\varphi_{\varepsilon_\nu}} \to e^{\varphi}$ almost everywhere. The functions are uniformally bounded. 
By the dominated convergence theorem, $e^{\varphi_{\varepsilon_\nu}} \to e^{\varphi}$ in $L^1(U)$.
Since the space of the psh functions is closed in $L^1_{\rm loc}(U)$, $e^{\varphi} \in PSH(U)$. Hence
$$i \d \dbar e^{\varphi} =e^{\varphi}( i \d \dbar \varphi + i\d \varphi \wedge \dbar \varphi) \geq 0$$
as a current.
For any compact set $K \subset U$, the mass of $i\d \varphi \wedge \dbar \varphi e^{\varphi} \wedge
\omega^{n-1}$ on $K$ is the mass of $i \d \dbar (e^{\varphi}) \wedge \omega^{n-1}$ on $K$ minus the mass of $i \d \dbar \varphi e^{\varphi}\wedge \omega^{n-1}$ on $K$ which is finite.
This means $\int_K |e^{\frac{\varphi}{2}} \d \varphi|^2_{\omega,h_0}dV_{\omega}$ is finite.
And it is closed with respect to the $L^2$ topology in the sense that considering a sequence $u_j,u \in L^2_{\loc}(U,\Lambda^{p,q}T^\star_Y\otimes L, \omega,h)$ such that
$u_j \to u$, we have by the inequality 
$$\int_U |\d \varphi \wedge u|_{\omega,h_0} dV_{\omega} = \int_U |\d \varphi e^{\frac{\varphi}{2}}|_{\omega,h_0}|u|_{\omega,h}dV_{\omega}$$
$$\leq (\int_U |e^{\frac{\varphi}{2}} \d \varphi|^2_{\omega,h_0}dV_{\omega})^{\frac{1}{2}} (\int_U |u|^2_{\omega,h}dV_{\omega})^{\frac{1}{2}}$$
which shows that $\d \varphi \wedge u_j \to \d \varphi \wedge u$ in $L^1_{\loc}(U,\Lambda^{p+1,q}T^\star_Y\otimes L, \omega,h_0)$, in particulier as currents.

We should mention that some similar discussion of the definition of covariant derivative with respect to a singular metric can also be found in \cite{Dem02}. (The author thanks Professor A.~H\"oring for mentioning the reference.)

\begin{myrem}{\rm 
We check here that the operator 
$$\d_h: L^2(X, \wedge^{n-q}T_X^* \otimes L,h) \to L^2(X, \wedge^{n-q+1}T_X^* \otimes L,h)$$
is a closed densely defined operator.

By a partition of unity argument, it is enough to check this on a local coordinate chart $U$. Assume that we have $h=e^{-\varphi}$ on $U$ for some psh function $\varphi$.
We claim that functions of the type $e^{(1/2+\varepsilon)\varphi}f$ with $\varepsilon>0$ small enough and $f$ smooth with compact support are in the domain of definition of $\d_h$ and are dense in $L^2(U, \wedge^{n-q}T_X^* \otimes L,h)$. In fact, we have
$$\d_h(e^{(1/2+\varepsilon)\varphi}f)=(3/2+\varepsilon) \d \varphi \wedge e^{(1/2+\varepsilon)\varphi}f+e^{(1/2+\varepsilon)\varphi}\d f.$$
Without loss of generality, we can assume that $\varphi$ is bounded from above.
Since $f, \d f$ are bounded and $|\d \varphi |^2 e^{\varepsilon \varphi}dV_\omega\leq\frac{1}{\varepsilon} i\partial\dbar(e^{\varepsilon\varphi})\wedge\omega^{n-1}$ is integrable, we have 
$\int_U |\d \varphi \wedge e^{(1/2+\varepsilon)\varphi}f|^2 e^{-\varphi}dV_\omega 
< \infty$ and
$\int_U |e^{(1/2+\varepsilon)\varphi}\d f|^2 e^{-\varphi}dV_\omega < \infty$.
Thus $e^{(1/2+\varepsilon)\varphi}f$ is in the domain of definition.

To prove the density, it is equivalent to show that smooth functions with compact support are dense in $L^2(U,e^{-\varepsilon \varphi}dV)$ where $dV$ is the Lebesgue measure. 
By \cite{Sko}, for $\varepsilon >0$ small enough, (e.g. such that $\varepsilon \sup_{x \in U} \nu(\varphi,x) <1$ where $\nu(\varphi,x)$ is the Lelong number of $\varphi$ at $x$),
$ e^{-\varepsilon \varphi}$ is locally integrable, thus $e^{-\varepsilon\varphi}dV_\omega$ is a locally finite measure.
Any real function $u \in L^2(U,e^{-\varepsilon \varphi}dV)$ can be approximated
in norm by a bounded function $\tilde u_\nu=\max(\min(u,\nu),-\nu)$, 
and then $\tilde u_\nu$ can be approximated by smooth compactly supported functions $u_\nu$ by taking the product of $\tilde u_\nu$ with a cut-off function and taking a convolution.

By the last paragraphs before the remark, if $u_\nu \to u$ in $L^2(e^{-\varphi})$ topology, then $\d_h u_\nu \to \d_h u$ in the weak topology of currents.
This shows that $\d_h$ is a closed operator by definition. }
\end{myrem}

Assuming for the moment that theorem 2 is valid, we infer theorem 3.
A consequence is that the inverse operator in the proof of the hard Lefschetz theorem is linear, a fact that is a priori non trivial.
\begin{proof}[Proof of theorem 3]

By theorem 2, we know that the morphism is surjective.
Since the morphism is the restriction of the wedge mulitplication operator on some subspace, it is linear.
Thus to show that it is a linear isomorphism, it is enough to show that it is injective.

Assume that $u \in H^0(X,\Omega_X^{n-q}\otimes L\otimes\cI(h))$ such that $\d_h u=0$ and $u \wedge \omega \equiv 0$ in $H^q(X,K_X \otimes L\otimes\cI(h))$.
It means that there exists $v \in L^2(X, \wedge^{n,q-1}T_X^* \otimes L, h)$ such that 
$$u \wedge \omega^q= \dbar v.$$
To prove that $u=0$, it is equivalent to prove that $u \wedge \omega^q=0$ by the pointwise Lefschetz isomorphism.
To prove that $u \wedge \omega^q=0$, it is enough to prove that $\parallel \dbar v \parallel =0$.

We have that
$$\parallel \dbar v \parallel^2=\int_X \langle \dbar v,u \wedge \omega^q  \rangle dV_\omega=\int_X \{ \dbar v,u \}.$$
On the other hand, we have that
$$\dbar \{v, u\}=\{\dbar v,u \}+(-1)^{n+q-1}\{v, \d_h u\}$$
since $v$ is a $(n,q-1)$ form.
By assumption that $\d_h u=0$, $\dbar \{v, u\}=\{\dbar v,u \}$.
Since $u$ is a $(n-q,0)$ form and $v$ is a $(n,q-1)$ form, by degree reason, we have that $\d \{v, u\}=0$.

Remark that $\{v,u\}$ is well defined a current (in fact $L^1_{\loc}$ with respect to any smooth metric on~$L$) since both $v,u$ are $L^2$ with respect to the singular metric $h$.

Thus by Stokes theorem, we have that 
$$\parallel \dbar v \parallel^2=\int_X d \{v,u\}=0.$$
\end{proof}
\section{Proof of theorem 2}
This section follows closely \cite{DPS01} with some additional estimates for the integral norms of the terms involved at each step. First, we reproduce the variant of the Bochner formula used in \cite{DPS01}.
\begin{myprop}
Let $(Y,\omega)$ be a {\rm complete} K\"ahler
manifold and $(L,h)$ a smooth Hermitian line bundle such that the
curvature current possesses a uniform lower bound $\Theta_{L,h}\ge -C\omega$.
For every measurable $(n-q,0)$-form $v$ with $L^2$ coefficients and
values in $L$ such that $u=\omega^q\wedge v$ has differentials
$\dbar u$, $\dbar^* u$ also in $L^2$, we have
$$
\Vert\dbar u\Vert^2+\Vert\dbar^*_h u\Vert^2=
\Vert\dbar v\Vert^2+\int_Y\sum_{I,J}\Big(\sum_{j\in J}\lambda_j\Big)|u_{IJ}|^2
$$
$($here, all differentials are computed in the sense of distributions$)$ and where $\lambda_1\le\cdots\le\lambda_n$ are the curvature
eigenvalues of $i \Theta_{L,h}$ expressed in an orthonormal frame
$(\partial/\partial z_1,\ldots,\partial/\partial z_n)$ (at some fixed
point $x_0\in Y$), in such a way that
$$
\omega_{x_0}=\ii\sum_{1\le j\le n}dz_j\wedge d\overline z_j,\qquad
(i\Theta_{L,h})_{x_0}=dd^c\varphi_{x_0}=
\ii\sum_{1\le j\le n} \lambda_jdz_j\wedge d\overline z_j.
$$
\end{myprop}

Now, $X$ denotes a compact K\"ahler manifold equipped with a K\"ahler
metric~$\omega$, and $(L,h)$ a pseudoeffective line bundle on~$X$. 
To fix the ideas, we first indicate the~proof in the
much simpler case when $(L,h)$ has a smooth metric $h$ (so that
$\cI(h)=\cO_X$), and then treat the general case (although it is not really used in the proof of the general case).
\myparagraph{}

Let $\{\beta\}\in H^q(X,\Omega^n_X\otimes L)$ be an arbitrary cohomology class. By standard Hodge theory, $\{\beta\}$ can be represented by a smooth
harmonic $(0,q)$-form $\beta$ with values in $\Omega^n_X\otimes L$. We can
also view $\beta$ as a $(n,q)$-form with values in $L$. The pointwise
Lefschetz isomorphism produces a unique $(n-q,0)$-form $\alpha$ such
that $\beta=\omega^q\wedge\alpha$. Proposition 1 
then yields
$$
\Vert\dbar\alpha\Vert^2+\int_X\sum_{I,J}\Big(\sum_{j\in J}\lambda_j\Big)
|\alpha_{IJ}|^2=\Vert\dbar\beta\Vert^2+\Vert\dbar^*_h \beta\Vert^2=0,
$$
and the curvature eigenvalues $\lambda_j$ are nonnegative by our assumption.
Hence $\dbar\alpha=0$ and $\{\alpha\}\in H^0(X,\Omega^{n-q}_X\otimes L)$
is mapped to $\{\beta\}$ by $\Phi^q_{\omega,h}=\omega^q\wedge \bullet$. 
\myparagraph{}
In this case, the proof of the closedness property of sections involves the identity
$$\int_X \{ \d_h v, \d_h v \}_h \wedge \omega^{q-1} =\int_X (\d \{ v, \d_h v \}_h - (-1)^{\deg v}\{ v, \dbar \d_h v \}_h)\wedge \omega^{q-1}.$$
Using the holomorphicity of $v$, the fact that $(X, \omega)$ is K\"ahler and the Stokes formula, we get
$$\RHS=(-1)^{\deg v+1}\int_X \{  v, -\d_h\dbar v+i\Theta_{L,h} v\}_h \wedge \omega^{q-1}=(-1)^{\deg v+1}\int_X \{  v, i\Theta_{L,h} v\}_h \wedge \omega^{q-1}$$
$$=-\int_X i\Theta_{L,h}\wedge \{  v,  v\}_h \wedge \omega^{q-1} \leq 0.$$
In the above calculation, we have used the formula
$$\d_h \dbar + \dbar \d_h = i\Theta_{L,h} \wedge \bullet.$$
The last inequality uses the curvature assumption. Therefore we have
$$\int_X \{ \d_h v, \d_h v \}_h \wedge \omega^{q-1} =0,$$
and this implies $\d_h v=0$.
\myparagraph{}
Let us return to the case of an arbitrary plurisubharmonic weight $\varphi$. We will need the following ``equisingular'' approximation of psh functions; here, equisingularity is to be understood in the sense that the multiplier ideal sheaves are preserved. A~proof can be found in \cite{DPS01} or \cite{psef}.
\begin{mythm}
Let $T=\alpha+dd^c\varphi$ be a closed
$(1,1)$-current on a compact Hermitian manifold $(X,\omega)$, where
$\alpha$ is a smooth closed $(1,1)$-form and $\varphi$ a quasi-psh
function. Let $\gamma$ be a continuous real $(1,1)$-form such that
$T\ge\gamma$. Then one can write
\hbox{$\varphi=\lim_{m\to+\infty}\widetilde\varphi_m$} where
\vskip3pt
\item{\rm (a)} $\widetilde\varphi_m$ is smooth in the complement $X\setminus Z_m$
of an analytic set $Z_m\subset X\,;$
\vskip3pt
\item{\rm (b)} $\{\widetilde\varphi_m\}$ is a non-increasing sequence, and $Z_m\subset Z_{m+1}$ for all~$m\,;$
\vskip3pt
\item{\rm (c)} $\int_X(e^{-\varphi}-e^{-\widetilde\varphi_m})dV_\omega$
is finite for every $m$ and converges to $0$ as $m\to+\infty\,;$
\vskip3pt
\item{\rm (d)} $($``equisingularity''$)$ $\cI(\widetilde\varphi_m)=\cI(\varphi)$ for all $m$
$\,;$
\vskip3pt
\item{\rm (e)} $T_m=\alpha+dd^c\widetilde\varphi_m$ satisfies $T_m\ge
\gamma-\varepsilon_m\omega$, where $\lim_{m\to+\infty}\varepsilon_m=0$.
\end{mythm}
Fix $\varepsilon=\varepsilon_\nu$ and let $h_\varepsilon=h_{\varepsilon_\nu}$
be an approximation of~$h$, such that $h_\varepsilon$ is smooth on
$X\setminus Z_\varepsilon$ ($Z_\varepsilon$ being an analytic subset of $X$),
$\Theta_{L,h_\varepsilon}\ge -\varepsilon\omega$,
$h_\varepsilon\le h$ and $\cI(h_\varepsilon)=\cI(h)$. As above we fix a reference smooth metric $h_0$ on $L$. We denote by $\beta$ the curvature form of $h_0$ and $h_{\varepsilon}=h_0 e^{- \varphi_{\varepsilon}}$ ($\varphi_{\varepsilon}$ is hence a global quasi-psh function on~$X$).
The existence of a such metric is guaranteed
by Theorem 5. Now, we can find a family
$$
\omega_{\varepsilon,\delta}=\omega+\delta(\ii\ddbar\psi_\varepsilon+\omega),
\qquad \delta>0
$$
of {\it complete K\"ahler} metrics on $X\setminus Z_\varepsilon$, where
$\psi_\varepsilon$ is a quasi-psh function on $X$ with analytic singualarity with
$\psi_\varepsilon=-\infty$ on $Z_\varepsilon$,
$\psi_\varepsilon$ smooth on $X\setminus Z_\varepsilon$ and $\ii\ddbar
\psi_\varepsilon+\omega\ge 0$ (see e.g.\ \cite{Dem82}, Th\'eor\`eme~1.5).
By construction, $\omega_{\varepsilon,\delta}\ge\omega$ and
$\lim_{\delta\to 0}\omega_{\varepsilon,\delta}=\omega$.
We look at the $L^2$ Dolbeault complex $K^\bullet_{\varepsilon,\delta}$
of $(n,\bullet)$-forms on $X\setminus Z_\varepsilon$, where the $L^2$ norms are
induced by $\omega_{\varepsilon,\delta}$ on differential forms and by
$h_\varepsilon$ on elements in~$L$. Specifically
$$
K^q_{\varepsilon,\delta}=\Big\{u{:}X\setminus Z_\varepsilon{\to}\Lambda^{n,q}
T^*_X\otimes L;\int_{X\setminus Z_\varepsilon}\kern-15pt
(|u|^2_{\Lambda^{n,q}\omega_{\varepsilon,\delta}
\otimes h_\varepsilon}+|\dbar u|^2_{\Lambda^{n,q+1}\omega_{\varepsilon,\delta}
\otimes h_\varepsilon})dV_{\omega_{\varepsilon,\delta}}<\infty\Big\}.
$$
Let $\cK^q_{\varepsilon,\delta}$ be the corresponding sheaf of germs
of locally $L^2$ sections on $X$ (the local $L^2$ condition
should hold on $X$, not only on $X\setminus Z_\varepsilon\,$!). Then,
for all $\varepsilon>0$ and $\delta\ge 0$,
$(\cK^q_{\varepsilon,\delta},\dbar)$ is a resolution of the sheaf
$\Omega^n_X\otimes L\otimes\cI(h_\varepsilon)=
\Omega^n_X\otimes L\otimes\cI(h)$. This is
because $L^2$ estimates hold locally on small Stein open sets, and the
$L^2$ condition on $X\setminus Z_\varepsilon$ forces holomorphic sections
to extend across~$Z_\varepsilon$ (\cite{Dem82}, Lemma 6.9).
\myparagraph{}
Let $\{\beta\}\in H^q(X,\Omega^n_X\otimes L\otimes\cI(h))$ be a
cohomology class represented by a smooth form with values in
$\Omega^n_X\otimes L\otimes\cI(h)$.
Then
$$
\Vert\beta\Vert_{\varepsilon,\delta}^2\le \Vert\beta\Vert^2=
\int_X|\beta|^2_{\Lambda^{n,q}\omega\otimes h}dV_\omega<+\infty.
$$
The reason is that $|\beta|^2_{\Lambda^{n,q}\omega\otimes h}dV_\omega$
decreases as $\omega$ increases, see e.g. \cite{Dem82}, Lemma 3.2. Now, $\beta$ is a $\dbar$-closed form in the Hilbert space
defined by $\omega_{\varepsilon,\delta}$ on $X\setminus Z_\varepsilon$ and for $\delta >0$, the K\"ahler metric is complete on $X\setminus Z_\varepsilon$, so there
is a $\omega_{\varepsilon,\delta}$-harmonic form
$u_{\varepsilon,\delta}$ in the same cohomology class
as $\beta$, such that
$$\Vert u_{\varepsilon,\delta}\Vert_{\varepsilon,\delta}\le
\Vert\beta\Vert_{\varepsilon,\delta}.$$
Let $v_{\varepsilon,\delta}$ be the unique $(n-q,0)$-form such that
$u_{\varepsilon,\delta}=v_{\varepsilon,\delta}\wedge
\omega_{\varepsilon,\delta}^q$ ($v_{\varepsilon,\delta}$ exists by the
pointwise Lefschetz isomorphism). Then
$$
\Vert v_{\varepsilon,\delta}\Vert_{\varepsilon,\delta}=
\Vert u_{\varepsilon,\delta}\Vert_{\varepsilon,\delta}\le
\Vert\beta\Vert_{\varepsilon,\delta}\le\Vert\beta\Vert.
$$
As $\sum_{j\in J}\lambda_j\ge -q\varepsilon$ by the assumption on
$\Theta_{L,h_\varepsilon}$, the Bochner formula for $X \setminus Z_{\varepsilon}$ yields
$$
\Vert\dbar v_{\varepsilon,\delta}\Vert_{\varepsilon,\delta}^2\le
q\varepsilon\Vert u_{\varepsilon,\delta}\Vert_{\varepsilon,\delta}^2
\le q\varepsilon\Vert\beta\Vert^2.
$$
But since $Z_{\varepsilon}$ is an analytic set, the integral can also be seen taken on $X$; In the following, we use it abusively.
These uniform bounds imply that there are subsequences $u_{\varepsilon,
\delta_\nu}$ and $v_{\varepsilon,\delta_\nu}$ with \hbox{$\delta_\nu\to 0$},
possessing weak-$L^2$ limits $u_\varepsilon=
\lim_{\nu\to+\infty}u_{\varepsilon,\delta_\nu}$
and $v_\varepsilon=\lim_{\nu\to+\infty}v_{\varepsilon,\delta_\nu}$.
The limit $v_\varepsilon=
\lim_{\nu\to+\infty}v_{\varepsilon,\delta_\nu}$ is with respect to
$L^2(\omega)=L^2(\omega_{\varepsilon,0})$. To check this, notice that
in bidegree $(n-q,0)$, the space $L^2(\omega)$ has the weakest topology
of all spaces~$L^2(\omega_{\varepsilon,\delta})$; indeed, an easy calculation
made in \cite{Dem82}, Lemma 3.2 yields
$$
|f|^2_{\Lambda^{n-q,0}\omega\otimes h}dV_\omega\le
|f|^2_{\Lambda^{n-q,0}\omega_{\varepsilon,\delta}\otimes h}
dV_{\omega_{\varepsilon,\delta}}\qquad
\hbox{if $f$ is of type $(n-q,0)$}.
$$
On the other hand, the limit
$u_\varepsilon=\lim_{\nu\to+\infty}u_{\varepsilon,\delta_\nu}$
takes place in all spaces $L^2(\omega_{\varepsilon,\delta})$, $\delta>0$,
since the topology gets stronger and stronger as $\delta\downarrow 0$
[$\,$possibly not in $L^2(\omega)$, though, because in bidegree $(n,q)$
the topology of $L^2(\omega)$ might be strictly stronger than that
of all spaces $L^2(\omega_{\varepsilon,\delta})\,$]. For fixed $\delta >0$, for any $\delta' < \delta$,we have 
$$\Vert u_{\varepsilon,\delta'}\Vert_{\varepsilon,\delta} \leq \Vert u_{\varepsilon,\delta'}\Vert_{\varepsilon,\delta'} \leq \Vert \beta \Vert$$
$$\Vert u_{\varepsilon}\Vert_{\varepsilon,\delta} \leq \liminf_{\delta' \to 0} \Vert u_{\varepsilon,\delta'}\Vert_{\varepsilon,\delta} \leq \Vert \beta \Vert$$
By Lebesgue's monotone convergence theorem, $u_{\varepsilon}$ is $L^2(\omega_{\varepsilon,\delta} \otimes h_{\varepsilon})$ bounded. The above estimates yield
$$
\Vert v_\varepsilon\Vert^2_{\varepsilon,0}=
\int_X|v_\varepsilon|^2_{\Lambda^{n-q,0}\omega\otimes h_\varepsilon}
dV_\omega\le\Vert\beta\Vert^2,$$
$$\Vert \dbar v_\varepsilon\Vert^2_{\varepsilon,0}\le q\varepsilon
\Vert\beta\Vert^2_{\varepsilon,0}=q\varepsilon \Vert\beta\Vert^2,$$
$$u_\varepsilon=\omega^q\wedge v_\varepsilon\equiv\beta
\qquad\hbox{in}~~H^q(X,\Omega^n_X\otimes L\otimes\cI(h_\varepsilon)).
$$
The last equality can be checked via the De Rham-Weil isomorphism, by using
the fact that the map
$\alpha\mapsto\{\alpha\}$ from the cocycle space
$Z^q(\cK^\bullet_{\varepsilon,\delta})$ equipped with its $L^2$ topology,
into $H^q(X,\Omega^n_X\otimes L\otimes\cI(h))$ equipped with its
finite vector space topology, is continuous. 
\myparagraph{}
For the closedness property, we want to control the $L^1_{\loc}$ norm of the covariant derivative with respect to the Lebesgue measure, which is well defined on $X$ since the metric is smooth outside an analytic set and the section is locally $L^2$ with respect to the metric. For any smooth $(n-q,0)$-form $v$ with compact support in $X \setminus Z_{\varepsilon}$, we can apply the Stokes formula to get
$$\int_X \{ \d_{h_{\varepsilon}} v, \d_{h_{\varepsilon}} v \}_{h_{\varepsilon}} \wedge \omega_{\varepsilon,\delta}^{q-1} =(-1)^{\deg v+1}\int_X \{  v, -\d_{h_{\varepsilon}}\dbar v+i\Theta_{L,h_{\varepsilon}} v\}_{h_{\varepsilon}} \wedge \omega_{\varepsilon,\delta}^{q-1}$$
$$=\int_X  (\dbar \{  v, \dbar v\}_{h_{\varepsilon}}-\{  \dbar v, \dbar v\}_{h_{\varepsilon}}-i\Theta_{L,h_{\varepsilon}} \wedge \{v,v\}_{h_{\varepsilon}} )\wedge \omega_{\varepsilon,\delta}^{q-1}$$
$$=\int_X  (-\{  \dbar v, \dbar v\}_{h_{\varepsilon}}-i\Theta_{L,h_{\varepsilon}} \wedge \{v,v\}_{h_{\varepsilon}} )\wedge \omega_{\varepsilon,\delta}^{q-1}.$$
We want to apply this identity to $v=v_{\delta,\varepsilon}$ that does not necessarily have compact support in $X \setminus Z_{\varepsilon}$.
However, the metric $\omega_{\varepsilon,\delta} \otimes h_{\varepsilon}$ is smooth and complete on $X\setminus Z_\varepsilon$, and this will allow us to extend the identity to $v=v_{\varepsilon,\delta}$. In fact, there exists a sequence of smooth forms $v_{\varepsilon,\delta, \nu}$ with compact support on $X\setminus Z_\varepsilon$ obtained by truncating $v_{\varepsilon,\delta}$ and by taking the convolution with a regularizing kernel, in such a way that $v_{\varepsilon,\delta,\nu} \to v_{\varepsilon,\delta}$ in $L^2(\omega_{\varepsilon,\delta} \otimes h_{\varepsilon})$ (and therefore in $L^2(\omega \otimes h_0)$ as well). For simplicity of notation, we put $\d_\varepsilon=\d_{h_\varepsilon}$ and denote by $\d_{\varepsilon,\delta}^*$ its dual with respect to the metric $\omega_{\varepsilon,\delta}\otimes h_{\varepsilon}$ (the latter operator depends on $\delta$, since the Hodge $*$ operator depends on the K\"ahler metric). By taking $v=v_{\varepsilon,\delta,\nu}$ in the above identity, neglecting the non positive term involving $\dbar v$ and using the curvature condition, we obtain
$$\Vert \d_{\varepsilon} v_{\varepsilon,\delta,\nu}\Vert^2_{\varepsilon,\delta}\le q\varepsilon
\Vert v_{\varepsilon,\delta,\nu} \Vert^2_{\varepsilon,\delta} .$$
Let us put $C=e^{max_X (\varphi_{\varepsilon_1})}$ (we have $C < \infty$ as $X$ is compact). Then by using $\omega_{\varepsilon,\delta} \geq \omega$, $h_{\varepsilon} \geq \frac{1}{C} h_0$, we get
$$\Vert \d_{\varepsilon} v_{\varepsilon,\delta,\nu}\Vert^2_{L^2(\omega \otimes h_0)}\leq C \Vert \d_{\varepsilon} v_{\varepsilon,\delta,\nu}\Vert^2_{\varepsilon,\delta},$$
By the Cauchy-Schwarz inequality and the fact that $X$ is compact and that the metrics $\omega$, $h_0$ are smooth, we find
$$\Vert \d_{\varepsilon} v_{\varepsilon,\delta,\nu}\Vert_{L^1(\omega \otimes h_0)}\leq C'\Vert \d_{\varepsilon} v_{\varepsilon,\delta,\nu}\Vert_{L^2(\omega \otimes h_0)}^{},$$
Since the covariant derivative is a closed operator and $v_{\varepsilon,\delta,\nu} \to v_{\varepsilon,\delta}$, $v_{\varepsilon,\delta} \to v_{\varepsilon}$ in $L^2(\omega_{\varepsilon,0} \otimes h_{\varepsilon})$, we conclude that
$$\Vert \d_{\varepsilon} v_{\varepsilon,\delta}\Vert_{L^1(\omega \otimes h_0)} \leq C'' \sqrt{q\varepsilon} \Vert\beta\Vert,$$
$$\Vert \d_{\varepsilon} v_{\varepsilon}\Vert_{L^1(\omega \otimes h_0)} \leq C'' \sqrt{q\varepsilon} \Vert\beta\Vert.$$
\myparagraph{}
Again, by arguing in a fixed Hilbert space $L^2(h_{\varepsilon_0})$ (since $\omega_{\varepsilon}=\omega$, the notation $L^2(h_{\varepsilon_0})$ will be used for fixed $\varepsilon_0 >0$), we find $L^2$ convergent subsequences $u_\varepsilon\to u$,
$v_\varepsilon\to v$ as $\varepsilon\to 0$, and in this way get
$\dbar v=0$ and
$$
\Vert v\Vert^2\le \Vert \beta\Vert^2,$$
$$u=\omega^q\wedge v\equiv \beta
\qquad\hbox{in}~~H^q(X,\Omega^n_X\otimes L\otimes\cI(h)).
$$
By closedness of the covariant derivative and by continuity of the injection $L^2(\omega \otimes h_0)\hookrightarrow L^1(\omega \otimes h_0)$ on the compact manifold $X$, we obtain
$$\Vert \d_{\varepsilon_0} v\Vert^2_{L^1(\omega \otimes h_0)} \leq C q\varepsilon_0 \Vert\beta\Vert^2.$$
As $\varphi=\lim_{\varepsilon \to 0}\varphi_{\varepsilon} $ and $ \d \varphi=\lim_{  \varepsilon \to 0}\d \varphi_{\varepsilon}$ in $L^1_{\loc}(h_0)$, and as we haven proven that $v$ is in fact holomorphic, by the continuity of the covariant derivative operator, we infer that $\d \varphi \wedge v=\lim_{\varepsilon \to 0} \d \varphi_{\varepsilon} \wedge v$ in the sense of distributions, and we have $\Vert \d_{h} v\Vert^2_{L^1(\omega \otimes h_0)}=0$, which means that $\d_{h} v =0$.
The closedness property is proved along the same lines.
\section{Foliation induced by sections}
We show that the closedness property of the holomorphic section provided by the hard Lefschetz theorem induces a foliation on $X$.
Here foliation means that there exists an irreducible analytic set $V$ of the total space $T_X$ such that for any $x \in X$, $V_x:=V \cap T_X$ is a complex vector space and the section sheaf $\cO(V) \subset \cO(T_X)$ is closed under the Lie bracket. It is equivalent to say that $\cO(V)$ is closed under the Lie bracket and that $\cO(T_X)/\cO(V)$ is torsion free.
 
We consider $v \in H^0(X, \Omega_X^{n-q} \otimes L \otimes \cI(h)), q \geq 1$ a parallel section with respect to the singular metric $h$.
In particular a section constructed by the hard Lefschetz theorem is such a section.  The interior product with $v$ gives an $\cO_X$-morphism (which is well defined on the whole of $X$ )
$$F_v:T_X \to \Omega_X^{n-q-1} \otimes L$$
$$X \mapsto \iota_X v.$$
We prove in the following that the kernel of $F_v$ defines a coherent subsheaf of $\cO(T_X)$ 
whose germs are closed under Lie brackets; this uses of course the closedness property.
Since the closedness under Lie brackets is a local property, we can take an open set $ U$ such that there exists a nowhere vanishing local generator  $s_L$ of the line bundle $L$ on $U$, and we verify the closedness of the Lie bracket on~$U$.
On $U$, $v=u \otimes s_L$ for some $u \in H^0(U,\Omega_X^{n-q})$. Denote by $X,Y$ two local tangent vector fields in $\ker F_v \subset \cO(T_X)$ defined on $U$. 
We have
$$0=d_h(u \otimes s_L)(X,Y,\bullet)$$
$$=(du \otimes s_L +(-1)^{\deg u}u \wedge d_h s_L)(X,Y,\bullet)$$
$$=du(X,Y,\bullet) \otimes s_L +(-1)^{\deg u}u \wedge d_h s_L(X,Y,\bullet)$$
$$=du(X,Y,\bullet) \otimes s_L +(-1)^{\deg u}[u(X,\bullet) d_h s_L(Y)-u(Y,\bullet)d_h s_L(X)+...]$$
$$=du(X,Y,\bullet) \otimes s_L$$
The above dots ... mean terms of the form $\pm u(X,Y,\bullet)d_h s_L(\bullet)$. The last equality uses of course the fact that $X,Y \in \ker F_v$.
Observe that $d_h(u \otimes s_L)$ is only almost everywhere defined (instead of pointwise defined).
The above equalities are calculated in the sense of currents.
\\
For any $X_0,...,X_{n-q}$ tangent vector fields of $U$ such that $X_0=X,X_1=Y$, we have
$$0=du(X_0,...,X_{n-q})=\sum_{i=0}^{n-q} (-1)^i X_i[u(X_0,...,\hat{X_i},...,X_k)]$$
$$+\sum_{0 \leq i <j \leq n-q} (-1)^{i+j}u([X_i,X_j],X_0,...,\hat{X_i},...,\hat{X_j},...,X_{n-q})$$
$$=-u([X,Y],X_2,...,X_{n-q}),$$
which means that $[X,Y] \in \ker(F_v)$.


Now we show that $\ker(F_v)$ is locally free over a Zariski open set.
For any $z \in X$, take an open neighborhood $V$ of $z$ such that $L|_V$ is trivial and on this open set $v(z)=\sum_{|I|=n-q} v_I(z)dz_I$ where $v_I \in \Gamma(V,\cO_X)$. 
Consider $\xi=\sum \xi_j(z) \frac{\d}{\d z_j}$ a local tangent vector field on $V$. For any multiindex $I$ and any $j\in I$, we write it in the form $I=(j,I'_j)$. Then $\xi \in \ker(F_v)$ if and only if $\sum_{j,I,|I|=n-q-1} \xi_j u_{(j,I)}dz_I=0$, i.e.\ if and only if for any $I,|I|=n-q-1$, $\sum_j \xi_j(z) u_{(j,I)}(z)=0$. This gives a local system of analytic equations defining $\ker(F_v)$. In particular, we see that $\ker(F_v)$ is 
locally free over the open set where the holomorphic linear system $\sum_j \xi_j(z) u_{(j,I)}(z)=0$ ($|I|=n-q-1$) has its generic rank.
In other words, $\ker(F_v)$ is locally free over the open set where the holomorphic matrix $( u_{(j,I)}(z)=0)_ {|I|=n-q-1}$ has its generic rank.
Then we have a regular foliation on this Zariski open set by the Frobenius theorem.
In particular, $\ker(F_v)$ is a holomorphic subbundle of the tangent bundle  on this Zariski open set.

Let $U'$ be the Zariski dense open set of $X$ such that $\ker(F_v)|_{U'}$ is a homolomorphic subbundle of the tangent bundle.
Define $V$ to be the Zariski closure of $\ker(F_v)|_{U'}$ in the total space $T_X$ of the tangent bundle. It is clear that $V$ is an irreducible analytic subset of $T_X$. In fact, $\ker(F_v)|_{U'}$ is contained in the regular part of $V$ as a complex space, thus the regular part $V$ is connected.
In~particular, $V$ has only one global irreducible component.

Observe that $\ker(F_v)$ coincides with $\cO(V)$ over $U'$.
For any local tangent fields $X,Y$, the image of $[X,Y]|_{U'}$ is contained in $V$.
Since $V$ is Zariski closed, the same holds for the image of $[X,Y]$ by passing to limit. In other words, $\cO(V)$ is closed under the Lie bracket,
and $\cO(V)$ defines a foliation on $X$.

To be more self-contained, we verify here that $T_X/ \cO(V)$ is torsion free.
Assume that $u \in \cO(T_X)_z$ and that $f \in \cO_{X,z}$ is such that $f \neq 0$, $fu \in \cO(V)_z$. We have to show that $u \in \cO(V)_z$.
Assume that $V$ is locally defined by $g_i(z, \xi)$.
By definition, $g_i(z, f(z)u_j(z))=0$ for every $i$, where $u_j$ is the components of $u$ in some local trivialization of $T_X$.
Since $V_x = V \cap T_{X,z}$ is a vector space, we have that
$g_i(z, u_j(z))=0$ for every $i$, which indeed means that $u \in \cO(V)_z $.

We can also reformulate our conclusion in the following form: denote by $r$ the generic rank of $\ker(F_v)$, then there is a meromorphic morphism
$$X \dashrightarrow\Gr(T_X,r)$$
$$z \mapsto \ker(F_{v,z})$$
where $\Gr(TX,r)$ is the
Grassmannian bundle of $r$-dimensional subspaces of $T_X$.

Let us observe that the foliation property only holds for the parallel sections.
In general, a non trivial section $v \in H^0(X, \Omega_X^{n-q} \otimes L), q \geq 1$, does not necessarily induce a foliation. We give below a concrete example of the non-integrability of $\ker(F_v)$ for such a section~$v$, and
thank Professor A.~H\"oring for pointing out the example.
It is interesting at this point to compare the situation with the following result proved in \cite{Dem02}: if $L$ is a psef line bundle over a compact K\"ahler manifold $X$ and $0\leq q\leq n=\dim X$, then for every
nonzero holomorphic section
$v \in H^0(X, \Omega_X^{q} \otimes L^{-1})$, the kernel $\ker(F_v)$
automatically defines a foliation on~$X$.

The example pointed out by A.~H\"oring first appeared in the paper
of Beauville \cite{Bea00}.
Let $A$ be an abelian surface and $X=A \times \P^1$.
Let $(U, V )$ be a basis of $H^0 (A, T_A )$ ,
and let $S, T$ be two vector fields on $\P^1$ which do not commute.
For example, in the homogenous coordinates $[w_1:w_2]$ of $\P^1$, we can take
$$S=w_2 \frac{\d}{\d w_1},~~ T=w_1 \frac{\d}{\d w_2}.$$
Then the vector fields $U + S$
and $V + T$ span a rank 2 subbundle $\Sigma$ of $T_X$.
Since $U+S$, $V+T$ have no common root, $\Sigma \cong \cO_X^{\oplus 2}$.
In particular, $\Sigma$ is not integrable, i.e.\
$\Sigma$ is not closed under the Lie bracket of vector fields.
Consider the short exact sequence of vector bundles
$$0 \to \Sigma \to T_X \to T_X/\Sigma \to 0.$$
We deduce that $T_X/\Sigma \cong -K_X$.
The quotient map $T_X \to T_X/\Sigma \cong -K_X$ induces by duality a vector bundle morphism $K_X \to \Omega^1_X$.
Thus we have a non trivial section $\eta_{
S,T} \in H^0(X, \Omega^1_X \otimes (-K_X))$.

To use the hard Lefschetz theorem, we take the following smooth metric on $-K_X$.
Denote by $\pi_1: X \to A$, $\pi_2: X \to \P^1$ the natural projections.
$-K_X = \pi_2^* \cO_{\P^1}(2)$.
Thus $-K_X$ is a semiample divisor.
By taking the smooth metric $h$ induced by a basis of global sections $\pi_2^* H^0(\P^1, \cO_{\P^1}(2))$ (or a base point free system of global sections), 
we get a smooth positive metric on $-K_X$.
In particular, the multiplier ideal sheaf associated to this metric is trivial. 
Moreover, by construction, the metric is real analytic.
In other words, we have a section $v \in H^0(X, \Omega_X^1 \otimes (-K_X))$ such that $\ker(F_v)$ is not integrable, while the metric is positive and
real analytic.

Fix any K\"ahler metric $\omega$ on $X$.
By the hard Lefschetz theorem, we have a surjective map
$$H^0(X, \Omega^1_X \otimes (-K_X)) \to H^2(X, \cO_X).$$
The image $\omega^2 \wedge \eta$ has a preimage $\eta$ which does not define a foliation on $X$.
\paragraph{}
Next, we derive by an explicit calculation what is the preimage given by the hard Lefschetz theorem, and show that this preimage indeed defines a foliation on $X$.
To simplify our exposition, we keep the same notation as above without assuming any longer that $S,T$ do not commute.
Fix $\omega_A$ a flat metric on $A$ such that $U,V$ form an orthonormal basis at each point.
Fix $\omega_{\P^1}$ a K\"ahler metric on $\P^1$ induced by the Fubini-Study metric and fix $\omega=\pi_1^* \omega_A+\pi^* \omega_{\P^1}$ a K\"ahler metric on $X$.
In particular, with this choice of metric, the induced metric $\wedge^3 \omega \otimes h$ on $K_X+(-K_X)$ is trivial.

We begin by showing that for any choice of $S,T$, the image $\omega^2 \wedge \eta_{S,T}$ is the same.
To verify this claim, we use the following isomorphism of $\C$-vector spaces.
Notice that $H^2(X, \cO_X) \cong \pi^*_1 H^2(A, \cO_A) \cong \C$.
Fix some $x \in \P^1$.
Consider the morphism
$$\iota: H^2(X, \cO_X) \to \C,$$
$$\{u\} \mapsto \int_{A \times \{x\}} u \wedge i U^* \wedge V^*.$$
Here $u\in C_{(0,2)}^{\infty}(X)$ is a representative of $\{u\} \in H^2(X, \cO_X)$.
The morphism $\iota$ is surjective since a generator of $H^2(X, \cO_X)$ can be represented by $\pi_1^* (\overline{U}^* \wedge \overline{V}^*)$, whose image is equal to $\int_A \omega^2_A>0$.
Since both sides are isomorphic to $\C$, we have an isomorphism.

For any $x \in \P^1$, let $W$ be a local generator $T_{\P^1}$ with norm 1 with respect to $\omega_{\P^1}$.
In particular, locally $U,V,W$ form an orthonormal basis with respect to $\omega$ pointwise.
Assume that locally \hbox{$S=fW$} and $T=g W$.
There exists a $C^{\infty}$ splitting of the short exact sequence 
$0 \to \Sigma \to T_X \to T_X/\Sigma \to 0$
by 
$T_X \cong \Sigma \oplus T_X/\Sigma$
which is induced by $\omega$.
Locally, $T_X$ is spanned by orthogonal basis $fU+gV-W$, $U+fW$ and $V+gW$. 
With this identification, $\eta$ can be locally given by for any $\xi \in T_X$
$$\eta(\xi)=\frac{\langle \xi, fU+gV-W \rangle}{|fU+gV-W|^2}(fU+gV-W).$$
Thus $\eta$ is given by
$$(\frac{f}{1+f^2+g^2} U^*+\frac{g}{1+f^2+g^2} V^*-\frac{1}{1+f^2+g^2} W^*) \otimes (fU+gV-W).$$
The anticanonical line bundle $-K_X$ is locally generated by
$$(fU+gV-W) \wedge (U+fW) \wedge (V+gW)=-(1+f^2+g^2) U\wedge V \wedge W.$$
In other words, the identification of $\Sigma^{\perp} \cong T_X/\Sigma \cong -K_X$ means the identification of $fU+gV-W$ with $-(1+f^2+g^2) U\wedge V \wedge W$. 
Thus $\omega^2 \wedge \eta$ seen as a $C_{(0,2)}^{\infty}$ form is given by
$$f \overline{V}^* \wedge \overline{W}^*+g \overline{U}^* \wedge \overline{W}^*+ \overline{U}^* \wedge \overline{V}^*.$$
Using this expression, $\iota(\omega^2 \wedge \eta_{S,T})$ is the same for any $S,T$.
Since $\iota$ is an isomorphism of vector spaces, $\omega^2 \wedge \eta_{S,T}$ is independent of the choice of $S,T$.

In the following, we show that the section constructed in the hard Lefschetz theorem for $\omega^2 \wedge \eta_{S,T}$ is $\eta_{S,T}$ associated with $S=T=0$.
We remark that since the metric is smooth, we can directly use the result of \cite{Eno93} without employing the equisingular approximation of \cite{DPS01}.
In other words, the preimage is given by the pointwise Lefschetz isomorphism of the harmonic representative of an element in $H^2(
X, \cO_X)$.

We claim that a generator of $H^2(X, \cO_X)$ can be represented by the harmonic $(0,2)$-form $\overline{U}^* \wedge \overline{V}^*$.
The reason is as follows.
Since the metric is trivial on $\cO_X$, the covariant derivative coincides with the exterior derivative.
Since $U,V$ are global parallel holomorphic sections, $dU^*=dV^*=0$.
This implies in particular that
$\dbar(\overline{U}^* \wedge \overline{V}^* )=0$.
On the other hand, $ \overline{U}^* \wedge \overline{V}^*$ is independent of the choice of coordinate on $\P^1$.
To prove that $\dbar^*(\overline{U}^* \wedge \overline{V}^* )=0$, it is enough to make a calculation in a normal coordinate chart centered at $x$.
In other words, locally $\omega=i U^* \wedge \overline{U}^*+i V^* \wedge \overline{V}^*+i W^* \wedge \overline{W}^*$ with $dW(x)=0$.
(The existence of the normal coordinate chart is ensured by the assumption that $\omega$ is K\"ahler.)
Since  $\dbar^*=- * \d *$, we have $\dbar^*(\overline{U}^* \wedge \overline{V}^* )(x)=0$, as this form involves only the value $dW(x)$ at $x$.
By the pointwise Lefschetz isomorphism, the preimage of $\overline{U}^* \wedge \overline{V}^*$ in the hard Lefschetz theorem is given by $ U^* \wedge V^*$.
It defines a foliation of $T_X$ generated by $U,V$, which has leaves $A \times \{x\}$ ($x \in \P^1$).
\section{Counterexample to coherence}
In this section, we wonder whether it is possible to replace the multiplier ideal sheaf by its ``lower semicontinuous regularization'', i.e.\
$$\cI_-(\varphi):= \bigcap_{\delta>0} \cI((1-\delta)\varphi),$$
which could be thought of as some sort of limit $\lim_{\delta\to 0_+}\cI((1-\delta)\varphi)$. A priori, as an infinite intersection of ideal sheaves, this lower semicontinous regularization might not be coherent.
It contains certainly $\cI(\varphi)$ and can be different from it if 1 is a jumping coefficient of the multiplier ideal sheaf. 
In this section, we show by a counterexample that the above infinite intersection $\bigcap_{\delta>0} \cI((1-\delta)\varphi)$ need  not be coherent for arbitrary psh functions; hence some further conditions should be added to ensure coherence and possible applications to algebraic geometry, thanks to Serre's GAGA theorem \cite{Ser}.

\begin{myprop}Let $B$ be the ball of radius $\frac{1}{2}$ centered at 0 in $\C^2$, and consider the plurisubharmonic function
$$\varphi(z,w)=\log |z|+\sum_{k \geq 1} \varepsilon_k \log(|z|+|w-a_k|^{N_k})$$
where $a_k$ is any sequence converging to 0 smaller than $\frac{1}{2}$ and
$\varepsilon_k>0$ and $N_k \in \N^*$ are suitable numbers $($ to be determined
later$)$. Then $\varphi$ defines multiplier sheaves such that
the intersection ideal
$\bigcap_{\delta>0} \cI((1-\delta)\varphi)$ is not coherent.
\end{myprop}  

The potential used above is a modification of the one given in \cite{GL} (and was suggested to the author by Demailly). Assume that the $a_k$'s are
distinct and not equal to zero. We recall the following elementary
calculation of \cite{Siu}.
\begin{mylem}
Let $a,b,$ and $c$ be some positive numbers such that $a$ and $c(1-\frac{\lceil a\rceil-a}{b})$ are not integers and $\lceil a\rceil-a < b <1$.
Let $p_0=\lceil a-1 \rceil$ and $q_0=\lfloor c(1-\frac{\lceil a\rceil-a}{b}) \rfloor$.
Then on $\C^2$, the multiplier ideal sheaf for the weight function 
$$a\,\log |z|+\log (|z|^b+|w|^c)$$
is generated by $z^{p_0+1}$ and $z^{p_0}w^{q_0}$.  Here $\lfloor \cdot \rfloor$ denotes the round-down and $\lceil \cdot \rceil$ denotes the round-up.
\end{mylem} 
Using this lemma, we can calculate the multiplier ideal sheaf at $(0,a_k)$ since near $(0,a_k)$ the function is equisingular to $\log |z|+ \varepsilon_k \log(|z|+|w-a_k|^{N_k})$. 
Using the trivial inequality 
$$\frac{1}{2}(\alpha^{\gamma}+\beta^{\gamma}) \leq (\alpha+\beta)^{\gamma} \leq 2^{\gamma}(\alpha^{\gamma}+\beta^{\gamma})$$
for $\alpha,\beta,\gamma$ non negative, one can easily reduce the required check to the lemma.
In order to compute the multiplier ideal sheaf associated to $(1-\delta)\varphi$ at $(0, a_k)$, $0<\delta<1$, we apply the lemma to $a=1-\delta$, $b=1-\delta$ and $c=(1-\delta)N_k \varepsilon_k$.
Once $\varepsilon_k, N_k$ are fixed, the number $c(1-\frac{\lceil a\rceil-a}{b})$ is an integer only for countably many values of $\delta$, a situation that does not affect $\cI_-(\varphi)$.
When $\varepsilon_k$ converge to 0 fast enough, $\varphi$ well define a psh function on $B$. In particular, we can choose $\varepsilon_k$ positive such that $\sum \varepsilon_k < \infty$.  
 By this assumption, $\varphi \geq (1+\sum \varepsilon_k) \log|z|$. Hence it is not identically infinite.
In particular, $\varphi$ is the limit of a decreasing sequence of psh functions $\log |z|+\sum_{k_0 \geq k \geq 1} \varepsilon_k \log(|z|+|w-a_k|^{N_k})$. 
Hence it is a psh function on $B$ for any choice of $N_k$.

Now fix $C >1$ and choose $N_k$ so that $N_k \varepsilon_k \geq C$ and $N_k \varepsilon_k$ is not an integer. Consider a given index~$k$.
For such a choice and $\delta$ small enough, $q_{k,\delta}= \lfloor N_k \varepsilon_k (1-2\delta) \rfloor \geq 1$.
By the lemma, $\cI((1-\delta)\varphi)$ is generated at $(0,a_k)$ by $z, (w-a_k)^{q_{k,\delta}}$. 
In particular, $(z, (w-a_k)^{\lfloor N_k \varepsilon_k \rfloor}) \subset (\cI_-(\varphi),a_k)$. 
Now we prove that $\cI_-(\varphi)$ is not coherent by contradiction.
If $\cI_-(\varphi)$ is coherent, since $B$ is a Stein manifold, by Cartan theorem A for any $(0, a_k)$ the map $H^0(B, \cI_-(\varphi)) \to \cI_-(\varphi)_{(0,a_k)}$ is surjective.
For any $f \in H^0(B, \cI_-(\varphi))$, $f(0,a_k)=0$ for any $k$. Since $(0,a_k)$ has a cluster point 0 on the complex line $\{z=0\}$, we have $f|_{\{z=0\}} \equiv 0$. 
In other words, $f$ can be divided by $z$.
But $(w-a_k)^{\lfloor N_k \varepsilon_k \rfloor}$ should then be the restriction of such a function $f$, and this contradiction yields the proposition.
\myparagraph{}

We check below that the coherence may however hold for psh functions that are not too badly behaved. By definition, it is enough to treat the case when $1$ is actually a jumping value of the multiplier ideal sheaves $t\mapsto\cI(t\varphi)$. First, we observe that when $\varphi$ has analytic singularity, we have $\cI_-(\varphi)=\cI((1-\delta)\varphi)$ for $\delta>0$ small enough,
in particular, $\cI_-(\varphi)$ is coherent. In fact, if $\varphi$ has the form $\varphi=\sum \alpha_j \log |g_j|$ where $D_j=g^{-1}_j(0)$ are nonsingular irreducible divisors with normal crossings, then $\cI(\varphi)$ is the sheaf of functions $f$ on open sets $U \subset X$ such that $\int_U |f|^2 \prod |g_j|^{-2\alpha_j}dV < \infty$. Since locally the $g_j$ can be taken to be coordinate functions from a local coordinate system $(z_1, \ldots ,z_n)$, the integrability condition is that $f$ be divisible by $\prod g_j^{m_j}$ where $m_j>\lfloor \alpha_j \rfloor$.
Hence $\cI(\varphi) =\cO(-\lfloor D\rfloor) =\cO(-\sum \lfloor \alpha_j\rfloor D_j)$. Saying that 1 is a jumping coefficient in this case means that there exist some index subset $J$ such that for any $j_0 \in J$ we have $\alpha_{j_0}=\lfloor \alpha_{j_0} \rfloor$.
In this case for $\delta$ small enough we have that
$$\cI((1-\delta)\varphi) =\cO(- \sum_{j \in J}(\alpha_{j}+1)D_{j}-\sum_{j \notin J} \lfloor \alpha_j\rfloor D_j)$$
and the conclusion follows. More generally, if $\varphi$ has arbitrary analytic singularity, there exists a smooth modification $\nu: \tilde{X} \to X$
of $X$ such that $\nu^* \cI(\varphi)$ is an invertible sheaf $\cO(-D)$
associated with a normal crossing divisor $D=\sum \lambda_j D_j$, where $(D_j)$ are the components of the exceptional divisor of $\nu$.
Now, we have $K_{\tilde{X}}=\nu^* K_X+R$ where $R=\sum \rho_j D_j$ is the zero divisor of the Jacobian determinant of the blow-up map. By the direct image formula, we get
$$\cI(\varphi)=\nu_* (\cO(R) \otimes \cI(\varphi \circ \nu)),$$
and the proof is reduced to the divisorial case.

Even more generally, for any psh function $\varphi$ and any psh function $\psi$ with zero Lelong numbers (i.e., for every $x$, $\nu(\psi,x)=0$), we have
$\cI(\varphi)=\cI(\varphi+\psi)$ (cf. Proposition 2.3 \cite{Kim}). By the above discussion we thus get $\cI_-(\varphi+\psi)=\cI((1-\delta)(\varphi+\psi))$ for $\delta>0$ small if $\varphi$ has analytic singularities.

In particular, when $X$ is 1-dimensional, Siu's decomposition theorem \cite{Siu74} can be used, to decompose $dd^c\varphi$ into the sum of a convergent series of Dirac masses and of a current with zero Lelong numbers; only the locally finite set of points where the Lelong number number is at least $1$ plays a role;
we then see that $\cI_-(\varphi)=\cI((1-\delta)\varphi)$
for $\delta$ small enough, hence $\cI_-(\varphi)$ is coherent.
More generally, the following variant of Nadel's proof on the coherence of multiplier ideal sheaf \cite{Nad90} can be exploited.

\begin{mylem}
For any psh function $\varphi$ on $\Omega \subset X$ such that $E_1(\varphi):=\{x; \nu(\varphi,x)\ge 1\}$ consists of isolated points, the sheaf $\cI_-(\varphi)$ is a coherent sheaf of ideals over $\Omega$. 
\end{mylem}
\begin{proof}
We follow the proof of Nadel. Without loss of generality, we can assume that $\Omega$ is the unit ball.
By the strong noetherian property of coherent sheaves, the family of sheaves generated by finite subsets of $H_-^2(\Omega,\varphi):=\{f \in \cO_{\Omega}(\Omega); \int_{\Omega}|f|^2 e^{-2(1-\delta)\varphi} < \infty, \forall \delta \in{} ]0,1[\}$ has a maximal element on each compact subset of $\Omega$, hence $H^2_-(\Omega,\varphi)$ generates a coherent ideal sheaf $\mathcal{J}$ in $\cO_{\Omega}$. 
By definition we have $\mathcal{J} \subset \cI_-(\varphi)$.
We will prove that in fact $\mathcal{J} = \cI_-(\varphi)$, which shows in particular that $\cI_-(\varphi)$ is coherent.

For the other direction, it is enough to prove that $\mathcal{J}_x+\cI_-(\varphi)_x \cap m^{s+1}_x=\cI_-(\varphi)_x$ for every integer~$s$, by the Krull lemma. 
Let $f\in \cI_-(\varphi)_x$ be defined in a neighborhood $V$ of $x$ and let $\theta$ be a cut-off function with support in $V$ such that $\theta= 1$ in some neighborhood of $x$. 
We solve the $\dbar$ equation $\dbar u=\dbar(\theta f)$ by  H\"ormander’s $L^2$ estimates ,with respect to the strictly psh weight 
$$\tilde{\varphi}(z) :=\varphi(z) + (n+s) \log|z-x|+|z|^2.$$
The integrability is ensured by the fact that $\dbar(\theta f)$ vanishes near $x$ and the Skoda integrability theorem \cite{Sko}.
We remark that the Lelong number outside a small open neighborhood of 0 is strictly less than 1 pointwise by the assumption that $E_1(\varphi)$ is isolated at $x$.

Hence we get a solution $u$ such that $\int_{\Omega}|u|^2e^{-2\varphi} |z-x|^{-2(n+s)}d\lambda < \infty$, thus $F=\theta f-u$ is holomorphic.
$F\in H^2_-(\Omega,\varphi)$ as a sum of a function in $L^2(\Omega,\varphi)$ and a function in $H^2_-(\Omega,\varphi)$. Moreover, $f_x-F_x=u_x \in \cI_-(\varphi)_x \cap m^{s+1}_x$.
This finishes the proof. 
\end{proof}
\section{On the optimality of multiplier ideal sheaves}

We study here whether the ideal sheaves $\cI(\varphi)$ involved
in the hard Lefschetz theorem can be replaced by ideals
$\cI((1-\delta)\varphi)\supset \cI(\varphi)$.
\forget{
In \cite{DPS01}, the following bundle valued hard Lefschetz is proven.
\begin{mythm}{\rm(\cite{DPS01})}\label{dps-th2}
Let $(L,h)$ be a pseudo-effective line bundle on a
compact K\"ahler manifold $(X,\omega)$ of dimension $n$, let
$\Theta_{L,h}\ge 0$ be its curvature current and $\cI(h)$ the
associated multiplier ideal sheaf. 
Assume that $L$ admits a smooth metric such that its curvature form $\alpha$ is semipositive.
Then, the wedge multiplication
operator $\omega^q\wedge\bullet$ induces a surjective morphism
$$
\Phi^q_{\omega,h}:
H^0(X,\Omega_X^{n-q}\otimes L\otimes\cI(h))\longrightarrow
H^q(X,\Omega_X^n\otimes L\otimes\cI(h)).
$$
\end{mythm}
To generalise this theorem, it is interesting to know whether we can change the multiplier ideal sheaf by some bigger ideal sheaves. (The ideal sheaf in general cannot be taken as the holomorphic function sheaves as shown in the counter example of \cite{DPS01}.)
In this section we give a counter example to give a negative answer for the following possible generalization.}
In other words, if $(L,h)$ is a pseudo-effective line bundle on a
compact K\"ahler manifold $(X,\omega)$ of dimension $n$,
$i\Theta_{L,h}\ge 0$ its curvature current and $\cI(h)$ the
associated multiplier ideal sheaf, we study whether
for any $\delta \in [0,1]$ small enough the wedge multiplication
operator $\omega^q\wedge\bullet$ induces a surjective morphism
$$
\Phi^q_{\omega,h}:
H^0(X,\Omega_X^{n-q}\otimes L\otimes\cI((1-\delta)h))\longrightarrow
H^q(X,\Omega_X^n\otimes L\otimes\cI((1-\delta)h)).
$$
First, we recall the following special case of the hard Lefschetz theorem.
\forget{Let $(L,h)$ be a pseudo-effective line bundle on a
compact K\"ahler manifold $(X,\omega)$ of dimension $n$, let
$\Theta_{L,h}\ge 0$ be its curvature current and $\cI(h)$ the
associated multiplier ideal sheaf.}
Assume that $L$ admits a smooth metric $h_0$ such that its curvature form $\alpha$ is semipositive.
Then, the wedge multiplication
operator $\omega^q\wedge\bullet$ induces a surjective morphism for any $\delta \in [0,1]$
$$
\Phi^q_{\omega,h}:
H^0(X,\Omega_X^{n-q}\otimes L\otimes\cI((1-\delta)h))\longrightarrow
H^q(X,\Omega_X^n\otimes L\otimes\cI((1-\delta)h)).
$$
The proof of this case just consists of applying the hard Lefschetz theorem to the Hermitian line bundle $(L, h_0^{\delta}h^{1-\delta})$.
If the line bundle admits a positive singular metric $h_0$ such that the corresponding Lelong numbers are equal to $0$ at every point, by Proposition 2.3 in \cite{Kim}, for any $\delta \in [0,1]$, the metric $(L, h_0^{\delta}h^{1-\delta})$ has a multiplier ideal sheaf equal to $\cI((1-\delta)h)$. Then the bundle valued hard Lefschetz theorem also implies the surjectivity property.

The condition that the line bundle admits a positive singular metric such that the Lelong number of this metric is pointwise 0 implies in particular
by regularization (see e.g. Theorem 14.12 in \cite{analmeth}) that the
line bundle is nef. However, the converse is false by example 1.7 in \cite{DPS94}, in which the only positive singular metric on the nef line bundle is the singular one induced by a section. An alternative example is given in \cite{Koi}: there, Koike considers the anticanonical line bundle $-K_X$ of the blow-up of $\P^2$ at 9 points, and shows that there exists some configuration of the nine points such that $-K_X$ is nef, while the singular metric with minimal singularities is induced by a section $s \in H^0(X, -K_X)\setminus \{0\}$.
In particular, there exists no singular metric on $-K_X$ with curvature $\geq 0$, such that the Lelong number of the singular metric is equal to 0 at each point.

This condition is also non equivalent to the semipositivity of the line bundle, although it is obviously implied by semipositivity. A counter example for the converse direction is provided by \cite{BEGZ}, example 5.4 and \cite{Kim07}, example 2.14. Take a non-trivial rank 2 extension $V$ of the trivial line bundle by itself, over an elliptic curve $C$, and an ample line bundle $A$ over~$C$. Then consider $X=\P(V \oplus A)$ and the associated line bundle $\cO(1)$. It is big and nef, and this is enough to conclude that it admits a semipositive singular metric with
Lelong numbers equal to 0. In fact, it is enough to argue for
the semipositive metric with minimal singularity.
By the Kodaira lemma, there exists $m_0 \in \N$ such that $\cO(m_0)=\tilde{A} +E$ where $\tilde{A}$ is an ample line bundle over $X$ and $E$ is an effective line bundle over $X$. 
For any $m \geq m_0$, a metric on $\cO(m)$ is induced by a smooth strictly positive metric on the ample line bundle $\tilde{A}+\cO((m-m_0))$ and by a singular metric induced by a non zero section on the effective line bundle $E$.
This metric itself induces a metric on $\cO(1)$ which is by definition more singular than the metric with minimal singularity. It has pointwise Lelong numbers at most equal to~$\frac{1}{m}$.
Hence the metric with minimal singularity has Lelong numbers equal to 0 pointwise.
However, $\cO(1)$ cannot admit a smooth semipositive metric: for this,
note that $X$ has a submanifold $Y \cong \P(V)$ given by the surjective bundle morphism $V \oplus A \to V$; a smooth semipositive metric on $\cO(1)$ would induce a smooth semipositive metric on $\cO_Y(1)$ by restriction, which is impossible by \cite{DPS94}.

As we have seen, the extension is possible if the minimal metric is not ``too bad''. This is also true in the purely exceptional case, as we will now see.

Let $X$ be the blow up a point of some smooth complex manifold $Y$ of
dimension~$n$. Denote by $E$ the exceptional divisor.
Let $L$ be a semi-positive line bundle on $X$ such that $L|_E$ is not trivial
on~$E$. Consider the line bundle $L+E$.
Take $h$ to be metric on $L+E$ induced by the canonical section of the effective divisor $E$, tensor product with the given semi-positive metric on $L$.
We start by remarking that for any $\delta \in{}]0,1]$ we have
$\cI((1-\delta)h)=\cO_X$. 
Hence the lower semicontinuous regularization of the multiplier ideal sheaf is trivial.
We claim that the map
$$H^0(X, \Omega_X^{n-q} \otimes L \otimes E) \to H^q(X, K_X \otimes L \otimes E)$$
is surjective for every $q \geq 1$.
First, by the hard Lefschetz theorem, we find that
$$H^0(X, \Omega_X^{n-q} \otimes L ) \to H^q(X, K_X \otimes L )$$
is surjective for every $q \geq 1$.
On the other hand, we have the following commutative diagram
$$
\begin{matrix}
H^0(X, \Omega_X^{n-q} \otimes L ) &\longrightarrow&
H^0(X, \Omega_X^{n-q} \otimes L \otimes E) \\
\noalign{\vskip5pt}
\Big\downarrow &&\Big\downarrow \\
\noalign{\vskip5pt}
H^q(X, K_X \otimes L ) &\longrightarrow& H^q(X, K_X \otimes L \otimes E).
\end{matrix}
$$
To show that the right arrow is surjective, it is enough to show that the bottom arrow is surjective.
By Serre duality, this is equivalent to proving that
$$H^{n-q}(X, -L-E) \to H^{n-q}(X, -L)$$
is injective.
By considering the long exact sequence associated to the short exact sequence
$$0 \to \cO_X(-L-E) \to \cO_X(-L) \to \cO(-L)|_E \to 0,$$
it is enough to show that for any $q\geq 1$
$$H^{n-q-1}(E, -L|_E)=0.$$
Remind that $E \cong \P^{n-1}$.
For any $q \in \Z$, for $0 < i< n-1$, we have that $H^i(\P^{n-1}, \cO(q))=0$.
Remind also that the Picard group of $\P^{n-1}$ is $\Z$.
This finishes the case $q \leq n-2$, and the case $q=n-1$ also
holds, since our assumptions $L\geq 0$ and $L|_E$ non trivial imply
$H^0(E, -L|_E)=0$.
The same arguments also work for $L=\cO_X$. We have an exact sequence
$$H^0(X, \cO_X) \to H^0(E, \cO_E) \to H^1(X, \cO(-E)) \to H^1(X, \cO_X).$$
The first morphism is an isomorphism -- it is just a restriction morphism applied to constant functions -- hence $H^1(X, \cO(-E)) \to H^1(X, \cO_X)$ is injective.
\vskip8pt

In general, as discussed in \cite{DPS4}, the minimal singular metric of a psef line bundle can still be very singular, and this fact might lead to a non coherent lower semicontinuous regularization of the multiplier ideal sheaf. It thus seems to be a difficult problem to improve the hard Lefschetz theorem by replacing the given multiplier ideal sheaf by its lower semicontinuous regularization, if~at all possible.

\end{document}